\documentclass[12pt,leqno,fleqn]{amsart}
\usepackage{hyperref}
\usepackage{amsmath,amstext,amsthm,amssymb} 

\usepackage{float}
\usepackage{amsrefs}
\hypersetup{
    colorlinks=true,       
    linkcolor=blue,          
    citecolor=magenta,        
    filecolor=magenta,      
    urlcolor=cyan,          
}

\usepackage{graphicx}

\allowdisplaybreaks
\swapnumbers

\allowdisplaybreaks
\swapnumbers

\theoremstyle{plain} 
\newtheorem{lemma}[equation]{Lemma} 
\newtheorem{proposition}[equation]{Proposition} 
\newtheorem{theorem}[equation]{Theorem} 
\newtheorem{corollary}[equation]{Corollary}

\theoremstyle{definition}
\newtheorem{definition}[equation]{Definition} 

\theoremstyle{remark}
\newtheorem{remark}[equation]{Remark}

\numberwithin{equation}{section}

\def\norm#1.#2.{\lVert#1\rVert_{#2}}
\def\Norm#1.#2.{\bigl\lVert#1\bigr\rVert_{#2}}
\def\NOrm#1.#2.{\Bigl\lVert#1\Bigr\rVert_{#2}}
\def\NORm#1.#2.{\biggl\lVert#1\biggr\rVert_{#2}}
\def\NORM#1.#2.{\Biggl\lVert#1\Biggr\rVert_{#2}}

\def\D{{\mathbb D}}

\def\ip#1,#2,{\langle #1,#2\rangle}
\def\Ip#1,#2,{\langle#1,#2\rangle}
\def\IP#1,#2,{\langle#1,#2\rangle}

\def\XXint#1#2#3{{\setbox0=\hbox{$#1{#2#3}{\int}$}
     \vcenter{\hbox{$#2#3$}}\kern-.5\wd0}}

\newcommand{\nat}{\mathbb{N}}

\newcommand{\T}{\mathbb{T}}

\newcommand{\LL}{\mathcal{L}}

\newcommand{\eps}{\varepsilon}
\newcommand{\ohne}{\backslash}
\begin{document}

\title[The class $B_\infty$]{The class $B_\infty$}

\author[A. Aleman]{Alexandru Aleman}
\thanks{}
\address{Centre for Mathematical Sciences, University of Lund, Lund, Sweden}
\email {aleman@maths.lth.se}

\author[S. Pott]{Sandra Pott}
\thanks{}
\address{Centre for Mathematical Sciences, University of Lund, Lund, Sweden}
\email {sandra@maths.lth.se}

\author[M.C. Reguera]{Maria Carmen Reguera}
\thanks{Supported partially by the Vinnova Grant 2014-01434, by Lund University, Mathematics in the Faculty of Science, and by VR grant 2015-05552}
\address{Centre for Mathematical Sciences, University of Lund, Lund, Sweden and School of Mathematics, University of Birmingham, Birmingham, UK}
\email {m.reguera@bham.ac.uk}


\begin{abstract}
We explore properties of the class of B\'ekoll\'e-Bonami weights $B_\infty$ introduced by the authors in a previous work. Although B\'ekoll\'e-Bonami weights are known to be ill-behaved because they do not satisfy a reverse H\"older property, we prove than when restricting to a class of  weights that are ``nearly constant on top halves'', one recovers some of the classical properties of Muckenhoupt weights. We also provide an application of this result to the study of the spectra of certain integral operators.
\end{abstract}

\maketitle

\section{Introduction}
Let $w$ be a weight on the unit disk $\mathbb{D}$,  that is,  a positive measurable function on $\mathbb{D}$.
Following B\'ekoll\'e and Bonami \cite{MR497663}, we say that $w$
belongs to the class $B_p$ for $1<p<\infty$, or equivalently satisfies the $B_p$ condition, if and only if

\begin{equation}
\label{e.bekolle}
B_{p}(w):=\sup_{\substack{I \text{ interval}\\ I\subset \T}}\left(\frac{1}{|Q_{I}|}\int_{Q_{I}}wdA\right) \left(\frac{1}{|Q_{I}|}\int_{Q_{I}}w^{1-p'}dA\right)^{p-1}<\infty.
\end{equation}
Here $Q_{I}$ denotes the Carleson box associated to interval $I\subset \T$,
$$
      Q_I= \{z=re^{i\alpha}: \,\, 1-|I|<r<1, |\theta-\alpha|<|I|/2\, \}.\}.
$$

In this paper we introduce some key properties of the limiting class $B_{\infty}$ of B\'ekoll\'e-Bonami weights,rpreviously introduced by the authors in \cite{APR}. B\'ekoll\'e-Bonami weights fail to satisfy a reverse H\"older property, which has prevented the development of a proper theory for the limiting class. On the other hand, a notion of $B_{\infty}$ appears naturally when looking at sharp estimates for the Bergman projection, as the authors discovered in \cite{APR}. In this paper we complete the picture of the $B_{\infty}$ theory when restricting attention to a very natural class of weights. 

\begin{definition}
We say that a  weight $w$ belongs to the class $B_{\infty}$, if and only if
\begin{equation}
\label{e.binfinity}
B_{\infty}(w):= \sup_{\substack{I \text{ interval}\\ I\subset \T}}\frac{1}{w(Q_{I})}\int_{Q_{I}} M(w1_{Q_{I}})<\infty,
\end{equation}
where $M$ stands for the Hardy-Littlewood maximal function over Carleson cubes and $Q_{I}$ denotes the Carleson cube associated to the interval $I\subset \T$ as defined in Definition \ref{d.carlbox}. 
\end{definition}
This definition of $B_{\infty}$ appears in earlier work of the authors on sharp estimates for the Bergman projection \cite{APR} and is motivated by a version of the Muckenhoupt condition
 $A_{\infty}$ introduced by Fujii \cite{MR0481968} and also studied by Wilson in \cites{MR972707, MR883661, MR2359017}. This $A_\infty$ definition appears in the recent works of Lerner \cite{MR2770437}, Hyt\"onen and P\'erez \cite{MR3092729} and Hyt\"onen and Lacey \cite{MR3129101} among others,
where it is used to find sharp estimates in terms of the Muckenhoupt $A_{p}$ and $A_{\infty}$ constants. Moreover, $B_{\infty}$ is a very natural class to provide sufficient conditions for two-weighted estimates for the Bergman projection in terms of the joint $B_{2}$ condition, as one can see from inspecting the work of the authors in \cite{APR}.

It is easy to see that in general weights in the $B_\infty$ and $B_p$ classes do not have the reverse H\"older property, as an 
arbitrarily small subset $E$ of a Carleson cube $Q_{I}$ can carry the entire weight. In this paper we restrict to a class of weights which are more tractable from this point of view.  These are weights which are almost constant on {\it $\rho$-tops} of Carleson boxes defined below.

\begin{definition} \label{d.tophalf}
Let $I$ be an interval on $\T$, and $0<\rho<1$. The $\rho$-top of $Q_{I}$ is the set
\begin{equation*}
T_{I,\rho}:=\{z=re^{i\alpha}: \,\, 1-|I|<r<1-(1-\rho)|I|, |\theta-\alpha|<|I|/2\, \}.\}.
\end{equation*} 
In the special case in which $\rho=1/2$, we will call $T_{I,1/2}$ the top half of $Q_{I}$ and we will denote it by $T_{I}$.
\end{definition}

In what follows we shall consider strictly positive weights $w$  such that there exists $0<\rho<1$,  and a constant $C_{\rho}$ such that for every interval $I$ on the boundary of $\D$  we have
\begin{equation}\label{e.regcon1}
C_{\rho}^{-1}w(\xi)\leq w(z)\leq C_{\rho}w(\xi) \quad \text{for all } z,\xi\in T_{I,\rho},
\end{equation}

Note   that if $w$ satisfies \eqref{e.regcon1} for some $\rho_0$, then it will satisfy this condition for all $\rho\in (0,1)$. This follows easily from the fact that if $1 > \rho >\rho_0$, then $T_{I,\rho}$ is contained in a finite union of $\rho_0$-tops $$\bigcup_{k=1}^NT_{J_k,\rho_0},$$ where $J_k\subset I, ~J_k\cap J_{k+1}\ne \emptyset$, and $N$ is independent of $I$.  This yields for $z,\xi\in T_{I,\rho}$
$$C_{\rho}^{-N}w(\xi)\leq w(z)\leq C_{\rho}^Nw(\xi).$$

Therefore, in what follows we shall refer to weights satisfying  \eqref{e.regcon1} without specifying the value of $\rho$.
It turns out that for such weights the class $B_\infty$ is very natural and enjoys similar properties as the analogous Muckenhoupt $A_\infty $ class. These properties are collected in the following theorem.
For a countable union of disjoint intervals $E = \cup_i I_i \subset \T$, we write $Q_E = \cup_i Q_{I_i}$.
\begin{theorem} \label{t.main} Let $w$ be a weight satisfying \eqref{e.regcon1}. Then the following are equivalent:
\begin{enumerate}
\item $w \in B_\infty$;
\item  There exists a constant $C>0$ such that
\begin{equation*}\label{e.logest}
\int_{Q_{I}} w(x) \log \left (e + \frac{w(x)|Q_{I}|}{w(Q_{I})} \right ) dx \leq C w(Q_{I}) 
\end{equation*}
for all Carleson cubes $Q_{I}$;
\item  For each $0 < \eps<1 $, there exists $\delta >0$ such that for any interval $I \subset \T$ and any countable union of disjoint intervals $ E \subset I $ with
$ |Q_E| / |Q_I| < \delta$, $ w (E) < \eps w(Q_I)$;
\item[$(3')$]  For each $0 < \eps<1 $, there exists $\delta >0$ such that for any interval $I \subset \T$ and any measurable subset $ \Omega  \subset Q_{I}$ with
$ |\Omega| / |Q_I| < \delta$, $ w (\Omega) < \eps w(Q_I)$;
\item $ w$ has the reverse H\"older property on Carleson cubes. That means, there exists $r>1$ and $C >0$ such that
$$
   \left( \frac{1}{|Q_{I}|}  \int_{Q_{I}}    w^r    \right)^{1/r}   \le C \frac{1}{|Q_{I}|}   \int_{Q_{I}}  w
$$   
for all Carleson cubes $Q_{I}$;
 
\item There exists $p>1$ such that $w\in B_{p}$;
\item There exists $C >0$ such that
$$
\left( \frac{1}{|Q_{I}|}\int_{Q_{I}}w(z)dA\right)\exp\left ( \frac{1}{|Q_{I}|}\int_{Q_{I}}\log(w^{-1})(z)dA(z)  \right ) \leq C.
$$

\end{enumerate}
Moreover, in (2) we can choose $C$ as $C' B_\infty(w)$, where $C'$ is an absolute numerical constant.
\end{theorem}

The proof of this theorem presents several difficulties. One is the lack of control of the weight by the maximal function of the weight, due to the geometric arrangement of Carleson boxes. This is a major obstruction to obtain a reverse H\"older property. But even if we had this control, $B_{\infty}$ weights lack a strong doubling property, characteristic of $B_{p}$ weights. We overcome these obstacles by using weights that are nearly constant on $\rho$-top halves.

Previous to this paper is the work of A. Borichev \cite{MR2036327}. Although not properly working on the limiting $B_{\infty}$ case, he considers self-improvement of the B\'ekoll\'e-Bonami $B_p$ class to a $B_{p-\epsilon}$ class. 
He obtains such an improvement when working with weights that are exponentials of subharmonic functions. Subsequently also weights that are constant on top halves appear in his argument. This self-improving property is classically associated with the reverse H\"older property and it is well-known for Muckenhoupt weights. Another paper which is close to the topic of this paper is the work of Duoandikoetxea et. al. \cite{MR3473651}, where properties of the $A_{\infty}$ class associated to general bases, for instance Carleson boxes or rectangles, are studied. Many of the implications that we prove in our paper for weights that are constant on top-halves have counterexamples in the general case of their paper.

This paper is organized as follows. In Section \ref{s.preelims} we list some properties of the weights satisfying  \eqref{e.regcon1}, and then proceed to  state precisely some of the definitions and describe some of the preliminaries needed for the proof of the main theorem. Section \ref{s.mainproof} contains the proof of the main theorem. In Section \ref{s.apply} we present an interesting characterization of the $B_\infty$ class with corresponding applications in the study of the spectra of certain integral operators. The last section contains the bibliography.

\section{Preliminaries} \label{s.preelims}

We begin with some remarks about the class of weights considered in this paper.\\

\begin{proposition}\label{p.regcond} (a) A differentiable strictly positive weight $w$ on $\D$ satisfies \eqref{e.regcon1}  if  there exists $C_w>0$ such that
	$$(1-|z|^{2})|\nabla w(z)|\leq C_{w}w(z), \quad z\in \mathbb D.$$	
	(b) If $w$ satisfies \eqref{e.regcon1}, then there exists a differentiable weight $\tilde{w}$ which satisfies \eqref{e.regcon1} and 
	$$K_w^{-1}w(z)\le \tilde{w}(z)\le K_ww(z)$$
	for some fixed constant $k_w>0$ and all $z\in \D$.\\
	(c)  If $w$ satisfies \eqref{e.regcon1}, then there exist constants $a_w,b_w>0$ such that 
	$$a_w^{-1}(1-|z|^2)^{b_w}\le w(z)\le a_w(1-|z|^2)^{-b_w},\quad z\in \D.$$
\end{proposition}
\begin{proof} (a) follows immediately from the inequality
	$$|\log w(z)-\log w(\xi) |\le  |z-\xi|\sup_{\zeta\in T_{I,\rho}}\frac{|\nabla w(\zeta)|}{w(\zeta)}, \quad z,\xi\in T_{I,\rho}.$$
	(b) If $u$ is a smooth positive function supported on $\{|z|<\frac1{2}\}$, with $$\int_{\mathbb{C}} udA=1,$$
	the weight $$\tilde{w}(z)=(1-|z|)^{-2}\int_{\mathbb{C}}u\left(\frac{z-\xi}{1-|z|}\right)w(\xi)dA(\xi)$$
	is differentiable  and satisfies 
	$$\inf\left\{w(\xi):~\frac{|z-\xi|}{1-|z|}<\frac1{2}\right\}\le \tilde{w}(z)\le \sup\left\{w(\xi):~\frac{|z-\xi|}{1-|z|}<\frac1{2}\right\},$$
	which easily implies the inequalities in (b).\\	
	Moreover, a direct estimate gives
	\begin{align*}|\nabla\tilde{w}(z)|&\lesssim (1-|z|)^{-1}\tilde{w}(z)+(1-|z|)^{-3}\int_{\mathbb{C}}|\nabla u|\left(\frac{z-\xi}{1-|z|}\right)w(\xi)dA(\xi)\\&
	\lesssim (1-|z|)^{-1}\tilde{w}(z)+(1-|z|)^{-1}w(z)\int_{\mathbb{C}}|\nabla u|dA.\end{align*}
	Together with the inequalities in (b) it follows that $\tilde{w}$ satisfies (a), hence it satisfies \eqref{e.regcon1}.
	
	\noindent
	(c) For differentiable weights $w$ satisfying \eqref{e.regcon1} we can integrate on rays from the origin to obtain
	$$\left|\log\frac{w(z)}{w(0)}\right|\le\int_0^{|z|} |\nabla w|(tz) dt,$$
	and the assertion follows by a direct calculation. The general case follows by (b).
\end{proof}

We also need the following definitions: 
\begin{definition} \label{d.carlbox}
Let $I$ be an interval $I\subset \T$ and let $e^{i\theta}$ be the center of $I$. We define the Carleson box $Q_I$ associated to $I$ as
$$
Q_{I}:=\{z=re^{i\alpha}: \,\,1-|I|<r<1, |\theta-\alpha|<|I|/2\, \}.
$$
\end{definition}


Throughout the paper, given an interval $I\subset \T$, we will denote by $\mathcal D(I)$ the set of dyadic descendants of $I$. The first descendants of $I$ will be the two disjoint intervals of size $2^{-1}|I|$, each of which contains exactly one end point of $I$. The remaining descendants will be defined recursively.  Also given a set $E\subset \D$, and an integrable function $f$, we write $f(E):=\int_{E} fdA(x)$.
We will need the following basic lemmas:

\begin{lemma}\label{l.strongd}
Let $w\in B_{\infty}$ with constant $B_{\infty}(w)$. Then there exists $\rho>0$ such that for every $I$ interval in $\mathbb T$, the $\rho$-top of $Q_{I}$, $T_{I,\rho}$, satisfies $w(T_{I,\rho})\geq \frac{1}{B_{\infty}(w)} w(Q_{I})$.
\end{lemma}

\begin{proof}
We introduce some notation for the proof. For fixed $I$, let $G_{0}:=T_{I}$, and more generally $\displaystyle G_{i}:=\cup_{\substack{J\in \mathcal D(I) \text{ dyadic } \\ |J|/|I|=2^{-i}}}T_{J}$. Notice that $Q_{I}=\cup_{i\geq 0} G_{i}$. The definition of the class $B_{\infty}$ implies 
\begin{equation} \label{e.sumb}
\sum_{i\geq 0}iw(G_{i})\leq B_{\infty}(w)\sum_{i\geq 0}w(G_{i}).
\end{equation}
Let us fix $i_{0}$ minimal such that $i_{0}\geq B_{\infty}(w)$. Using \eqref{e.sumb} we have
$$
\sum_{i\geq 0}w(G_{i})\leq \sum_{i=0}^{i_{0}}iw(G_{i})+\sum_{i>i_{0}}(i-B_{\infty}(w))w(G_{i})\leq B_{\infty}(w)\sum_{i\geq i_0}w(G_{i}).
$$
Now there exists $0<\rho<1$ depending on $B_{\infty}(w)$ but not on $I$ such that $\cup_{i=0}^{i_0}G_i\subset T_{I,\rho}$ and hence
$ w(T_{I,\rho})\geq \sum_{i\geq i_0}w(G_{i})\geq \frac{1}{B_{\infty}(w)}w(Q_{I})$ as desired.
\end{proof}

Given $I\subset \T$, a weight $w$ and $R>4$, we use a corona decomposition to define a collection of cubes $\mathcal L$ (that depends on the choice of the initial $I$, $w$ and $R$) as follows:

\begin{enumerate}
\item Firstly, we define the stopping children of a given interval $I$, $\mathcal L(I)$:
$$
\mathcal L(I):= \left \{\text{ maximal intervals } L\in \mathcal D(I) \text{ such that } \frac{w(Q_{L})}{|Q_{L}|}>R \frac{w(Q_{I})}{|Q_{I}|} \right \}
$$
\item Iterating this stopping procedure we construct collections of intervals $\mathcal L_{1}:= \mathcal L(I)$ and in general for $j\geq 2$, $\mathcal L_{j}:=\cup_{L\in \mathcal L_{j-1}} \mathcal L(L)$. We denote $\mathcal L:= \cup_{j\geq 1} \mathcal L_{j}$.
\end{enumerate}

A couple of remarks are in order. First, from the stopping procedure we obtain that given $L\in \mathcal L_{j}$ ,
\begin{equation}\label{e.stopsparse}
        \sum_{\substack{L' \in \LL_{j+1}\\ L'\subset L}}    |Q_{L'}| \le R^{-1} |Q_{L}|.
\end{equation}
Second, by maximality one obtains
$$ 
      R^j \frac{w(Q_{I})}{|Q_{I}|} < \frac{w(Q_{L})}{|Q_{L}|} \le (4R)^{j}  \frac{w(Q_{I})}{|Q_{I}|} 
$$
for all $L \in \LL_j$, $j  \in \nat$.
Finally, let us consider the dyadic maximal function on Carleson cubes associated to $\mathcal D (I)$, $M_{d}$, where for $x\in\D$
$$
M_{d}(f)(x):= \sup_{\substack{x\in Q_{J}\\ J\in \mathcal D (I)}}\frac{1}{|Q_{J}|}\int_{Q_{J}} f(y)dA(y).
$$
Then
\begin{equation}\label{e.mlinear}
    M_{d}(1_{Q_{I}} w)  \lesssim  \left (1 + \sum_j \sum_{L\in \LL_j}(4R)^j 1_{Q_{L}}\right )  \frac{w(Q_{I})}{|Q_{I}|}.
\end{equation}

The following observation is crucial in the proof of the main theorem.

\begin{remark}
Let $z\in Q_{I}$ for some $I\subset \T$, then there exists $J\in \mathcal D(I)$ such that $z\in T_{J}$, where we use the notation from Def. \ref{d.tophalf}. Using \eqref{e.regcon1}, we have
$$
w(z)\lesssim \frac{w(T_{J})}{|T_{J}|}\lesssim \frac{w(Q_{J})}{|Q_{J}|},
$$
and we conclude that
\begin{equation}\label{e.maxcontrol}
w(z)\lesssim  M_{d}(w)  (z),
\end{equation}
with constants only depending on the constant $C_{1/2}$ in \eqref{e.regcon1}.
\end{remark}

\section{Proof of Theorem \ref{t.main}} \label{s.mainproof}

\begin{proof}
We first notice that the equivalence of $(1)$ and $(2)$ is not true in general, see Counterexample 4 in \cite{MR3473651}. One of the directions always holds, namely 
$(2) \Rightarrow (1)$. Since Carleson cubes form what is known as a Muckenhoupt basis, the proof of this implication can be found in \cite{MR3473651}. We prove the opposite implication, where the use of weights satisfying \eqref{e.regcon1} is crucial. 

Let $Q_{I}$ be a Carleson cube. Without loss of generality, we can assume that $\frac{w(Q_{I})}{|Q_{I}|}=1$. We consider the dyadic grid associated to $I$, $\mathcal D(I)$, the maximal function $M_{d}$, and the corona decomposition $\mathcal L$ of $\mathcal D(I)$ with $R>4$ as described in section 2. We write $\mathcal{L}_0$ for the collection $\{Q_I\}$.  If $L\in \mathcal L_{j}$ and $x\in Q_L\setminus \cup_{\substack{L'\in \mathcal L_{j-1}\\ L'\subset L} }Q_{L'}$, then  
$$
M_{d}w(x)\approx \frac{w(Q_{L})}{|Q_{L}|} \approx 2^j \frac{w(Q_{I})}{|Q_{I}|}=2^j
$$ and 
$$\log \left (e + M_{d}w(x)\right )\approx j+1.
$$
Using (1) and the estimates above, we conclude
\begin{align*}
 B_\infty(w) w(Q_{I})& \geq \int_{Q_{I}}M(w1_{Q_{I}})(x)dA(x) \geq \int_{Q_{I}}M_{d}(w1_{Q_{I}})(x) dA(x)\\
 & \gtrsim  \int_{Q_{I}}\sum_{j \ge 0}\sum_{L\in \mathcal L_{j}}\frac{w(Q_{L})}{|Q_{L}|}1_{Q_{L}}dA(x)\\
& = \int_{Q_{I}} w(x) \sum_{j \ge 0}\sum_{L\in \mathcal L_{j}}1_{Q_{L}}dA(x)\\
& = \int_{Q_{I}} w(x) \left (\sum_{j\geq 1} j \sum_{L\in \mathcal L_{j}} 1_{Q_{L}\setminus \cup_{\substack{L'\in \mathcal L_{j+1}\\ L'\subset L} }Q_{L'}} +1\right )dA(x)\\
& \approx \int_{Q_{I}} w(x) \log \left (e + M_{d}w(x)\right )dA(x)\\
& \gtrsim \int_{Q_{I}} w(x) \log \left (e + w(x)\right )dA(x),
\end{align*}
where  we have used that $w$ satisfies \eqref{e.regcon1} and \eqref{e.maxcontrol} in the last inequality.
The implications $(2) \Rightarrow  (3')$ and $(3') \Rightarrow  (1)$ correspond to Theorem 4.1 and 6.1 in \cite{MR3473651} and we will not include them here.

To prove the equivalence of $(4)$, $(3')$ and $(3)$, first note that clearly
$(4) \Rightarrow (3') \Rightarrow (3)$ by H\"older's inequality. The proof of the reverse implication $(3) \Rightarrow (4)$ runs along the lines of  
Theorem 3.3 in Wilson \cite{MR2359017}, page 46. Let us fix a Carleson cube $Q_{I}$. Choose $\delta>0$ from (3) for $\epsilon=1/5^{r}$. Now consider the corona decomposition $\mathcal L$ with $R=\frac{1}{\delta}$ as defined in Section 2. For this choice of $R$, \eqref{e.stopsparse} gives

$$
\sum_{\substack{L' \in \LL_{j+1}\\ L'\subset L}}    |Q_{L'}| \le \delta |Q_{L}|,
$$
and thus by the definition of $\delta$ together with (3)
 $$
  \sum_{L' \in \LL_{j+1}} w(Q_{L'})  \le     \frac{1}{5^{r}} \sum_{L \in  \LL_{j} } w(Q_{L}) \le \frac{1}{5^{r(j+1)}} w(Q_{I}),
 $$
 where the second inequality is obtained by iterating the argument. Using (\ref{e.mlinear}),
we thus estimate 
\begin{eqnarray*}
 \frac{1}{|Q_{I}|}    \int_{Q_{I}} w^r 
 &\lesssim&  \frac{1}{|Q_{I}|}    \int_{Q_{I}} (M_{d}  (1_{Q_{I}} w))^r \lesssim      \frac{ w(Q_{I})^r}{|Q_{I}|^{1+ r}}    \int_{Q_{I}}   (1 +  \sum_{k\ge 1}  \sum_{L \in \LL_k(Q_{I})}  (4R)^k 1_{Q_{L}})^r    \\
    & \approx &
       \frac{ w(Q_{I})^r}{|Q_{I}|^{1+ r}}       \int_{Q_{I}}  \left( 1 +  \sum_{k\ge 1}  \sum_{L \in \LL_k(Q_{I})} (4R)^{rk} 1_{Q_{L}}   \right)  \\
      & =& \left(\frac{w(Q_{I})}{|Q_{I}|}\right)^r   +   \frac{w(Q_{I})^r}{|Q_{I}|^{1+r}}\  \sum_{k \ge 1}   (4R)^{rk}  \sum_{L \in \LL_k(Q_{I})}  |Q_L|  \\
      & \le &\left(\frac{w(Q_{I})}{|Q_{I}|}\right)^r   +   \frac{w(Q_{I})^{r-1}}{|Q_{I}|^{r}}\  \sum_{k\ge 1}   4^{rk} R^{(r-1)k} \sum_{L \in \LL_k(Q_{I})}  w(Q_L)    \\
       &\le& \left(\frac{w(Q_{I})}{|Q_{I}|}\right)^r   +   \frac{w(Q_{I})^r}{|Q_{I}|^{r}}\  \sum_{k \ge 1} \frac{1}{5^{rk}}  4^{rk} R^{(r-1)k}   \\
       & \le & \left(\frac{w(Q_{I})}{|Q_{I}|}\right)^r  (1 +  \sum_{k \ge 1}  \frac{1}{5^{rk}}  (4^r R^{(r-1)})^k   \le C_r^r  \left(\frac{w(Q_{I})}{|Q_{I}|}\right)^r,
\end{eqnarray*}
provided $r$ is chosen such that
$  \frac{4^r R^{1-r}}{5^r}  <1 $.

We now prove $(3) \Rightarrow (5)$. First, we notice that (3) implies a doubling condition on the weight. The proof is similar to the classical proof in the case of Muckenhoupt weights,
but we have to take some care to adapt it to our setting. 

\begin{definition} We say that the weight $w$ is doubling, if there exists a constant $C>0$ such that
$$
     w(Q_{2I}) \le C w(Q_I).
$$
In particular, this implies that $w(Q_{\tilde I}) \le C^2 w(Q_I)$, where $ \tilde I$ is the dyadic parent of $I \in \mathcal{D}$.
\end{definition}

Choose $\delta<1$ corresponding to $\epsilon=1/2$ in (3) and choose $\rho$ such that $ 1 > \rho \ge   \max\{1- \delta, 3/4\}$, and such that $1/4$ is an integer multiple of $1 - \rho$. Let $T_{2I, \rho}$ be as in Definition \ref{d.tophalf},  and let $C_\rho$ be the constant from (\ref{e.regcon1}).
Then 
$$
    w(T_{2I, \rho}) \le C_\rho^2  w(T_{2I, \rho} \cap Q_I ) \frac{|T_{2I, \rho}|}{|T_{2I, \rho}\cap Q_I |} \le 6 \, C_\rho^2 w(Q_I).
$$
Now let us consider the remainder $Q_{2I} \ohne (Q_I \cup T_{2I, \rho} )$. For any union $E$ of countably many disjoint intervals of length less than $2(1 - \rho)|I|$ contained in $2I \ohne I$, we have
$$
    |Q_E| <  2 |I| (1 -\rho)  | I|   < \delta |Q_{2I}|
$$
and therefore by (3)
$$
    w(Q_E) < \frac{1}{2} w(Q_{2I})
$$
 Taking the supremum of such unions $E$ and using the fact that    $1/2 |I|$ is an integer multiple of $2(1 - \rho)|I|$, we obtain
 $$
       w(Q_{2I} \ohne (Q_I \cup T_{2I, \rho} )) \le \frac{1}{2} w(Q_{2I})
 $$
 and therefore
 $$
    w(Q_{2I}) \le 2 w(Q_I \cup T_{2I, \rho}) \le (6  C_\rho^2 +1) w(Q_I).
 $$
 Hence $w$ is doubling.

  In particular, there exists a constant $C >0$ such that 
  \begin{equation}   \label{dyad.doub}
  w(Q_{I})\leq C^{N} w(Q_{I^{(N)}}) \text{ for any  dyadic  $N$-th descendant $I^{(N)}$ of $I$}.
  \end{equation}
We use this fact to state the following lemma.

\begin{lemma} \label{l.doubcons}
Let $w$ be a doubling weight. Then for any $1>\eta>0$, there exist $1>\tau>0$ such that for any interval $I \subseteq \mathbb T$ and any set $E$ that is a countable union of disjoint intervals in $\mathcal D(I)$ such that $w(Q_{E})<\tau w(Q_I)$, one has $|Q_E|<\eta |Q_{I}|$.
\end{lemma}

\begin{proof}
Let $C$ be the dyadic doubling constant of $w$ as in (\ref{dyad.doub}). Let us fix $0<\eta<1$, then there exists a natural number $N$ such that $2^{-2(N+1)}< \eta \leq 2^{-2N}$. Consider $\tau<\frac{1}{C^{N}}$, let $I$ be an interval and $E$ a countable union of intervals contained in $I$ with $w(Q_{E})<\tau w(Q_{I})$. By the doubling property, $Q_{E}$ cannot contain the Carleson box associated to any $N$-th descendant of $I$. Thus
$|Q_{E}|\leq 2^{-2(N+1)}|Q_{I}| < \eta|Q_{I}|$, concluding our proof.
\end{proof}

Using (\ref{e.regcon1}) and  Lemma \ref{l.strongd}, we have that if $z\in T_{\rho,I}$ for some interval $I$,
\begin{equation} \label{e.maxcon}
w(z)\approx \frac{w(T_{\rho, I})}{|T_{\rho, I}|}\gtrsim \frac{w(Q_I)}{|Q_I|},
\end{equation}
with constants only depending on $C_\rho$ and $B_\infty(w)$.

We will also need the following stopping decomposition. Let $A$ be a constant so that $A>\frac{1}{\tau}$, where $\tau$ is as in the Lemma \ref{l.doubcons}. Given a Carleson box $Q_{0}$, we define
$$
\mathcal L(Q_{0}):=\left \{ \text{ maximal dyadic Carleson boxes } Q_{I}\subset Q_{0}:\,\, \frac{w(Q_I)}{|Q_{I}|}<A^{-1}\frac{w(Q_{0})}{|Q_{0}|} \right \}.
$$

Then 
$$
\sum_{Q_{I}\in \mathcal L(Q_{0})}w(Q_{I})\leq A^{-1}w(Q_{0})<\tau w(Q_{0}),
$$
and by Lemma \ref{l.doubcons},
$$
\sum_{Q_{I}\in \mathcal L(Q_{0})}|Q_{I}|\leq \eta |Q_{0}|.
$$

We define $\mathcal L_{1}:=\mathcal L(Q_{I_{0}})$ and more generally, $\mathcal L_{k}:=\cup_{Q_{I}\in\mathcal L_{k-1}} \mathcal L(Q_{I})$ for $k >1$. We also define $\mathcal L=\cup_{k\geq 1} \mathcal L_{k}$. Given $Q_L\in \mathcal L$, we define $\mathcal D(Q_L)$ as the set of dyadic Carlesson boxes that have $Q_L$ as their stopping father. We have the following properties:

For $k\geq 1$, $Q_I\in \mathcal L_{k}$ and $Q_J\in \mathcal D(Q_I)$,
\begin{equation}\label{e.nonstop}
\frac{w(Q_J)}{|Q_{J}|}<A^{-1}\frac{w(Q_{I})}{|Q_{I}|},
\end{equation}
and 
\begin{equation*}
\frac{w(Q_J)}{|Q_{J}|}>\frac{C}{4}A^{-1}\frac{w(Q_{I})}{|Q_{I}|},
\end{equation*}
where $C$ is the dyadic doubling constant.

Hence by iteration
\begin{equation} \label{e.cabove} \frac{w(Q_I)}{|Q_{I}|}<A^{-k}\frac{w(Q_{0})}{|Q_{0}|},
\end{equation} 
\begin{equation} \label{e.cbelow}
\frac{w(Q_I)}{|Q_{I}|}> \left(\frac{C}{4}\right)^{-k} A^{-k}\frac{w(Q_{0})}{|Q_{0}|},
\end{equation}
and
\begin{equation} \label{e.ceta}
\sum_{Q_{I}\in\mathcal L_{k}}|Q_{I}|\leq \eta^{k}|Q_{0}|.
\end{equation}

We have now all the ingredients to complete the proof:
\begin{eqnarray*}
\frac{1}{|Q_{I_{0}}|}\int_{Q_{I_{0}}}w(z)^{-\varepsilon}dA(z) & = &
\frac{1}{|Q_{I_{0}}|}\sum_{Q_L\in \mathcal L}\sum_{Q_ I\in \mathcal D(Q_L)}\int_{T_{I}}w(z)^{-\varepsilon}dA(z)\\
 &\leq & \frac{1}{|Q_{I_{0}}|}\sum_{Q_L\in \mathcal L}\sum_{Q_ I\in \mathcal D(Q_L)}\int_{T_{\rho, I}}w(z)^{-\varepsilon}dA(z)\\
& \approx & \frac{1}{|Q_{I_{0}}|}\sum_{Q_L\in \mathcal L}\sum_{Q_I\in \mathcal D(Q_L)}\left ( \frac{w(T_{\rho, I})}{|T_{\rho, I}|} \right)^{-\varepsilon}|T_I|\\
&\leq & C(A,\epsilon,B_{\infty}(w)) \frac{1}{|Q_{I_{0}}|}\sum_{Q_L\in \mathcal L}\left ( \frac{w(Q_{L})}{|Q_{L}|} \right)^{-\varepsilon}|Q_L|\\
&\leq & C(A,\epsilon,B_{\infty}(w)) \frac{1}{|Q_{I_{0}}|} \left ( \frac{w(Q_{I_{0}})}{|Q_{I_{0}}|} \right)^{-\varepsilon}\sum_{k}\left(\frac{4}{C}\right)^{\varepsilon k}A^{\varepsilon k}\sum_{L\in \mathcal L_{k}} |Q_L|\\
&\leq & C(A,\epsilon,B_{\infty}(w)) \left ( \frac{w(Q_{I_{0}})}{|Q_{I_{0}}|} \right)^{-\varepsilon} \sum_{k} \left(\frac{4}{C}\right)^{\varepsilon k}A^{\varepsilon k}\eta ^{k}\\
&\leq & \tilde{C}(A,\epsilon,B_{\infty}(w)) \left ( \frac{w(Q_{I_{0}})}{|Q_{I_{0}}|} \right)^{-\varepsilon},
\end{eqnarray*}
where the geometric series in the penultimate line converges when choosing $\varepsilon$ sufficiently small. This concludes the proof of  $(2) \Rightarrow (5)$.

For the proof of  $ (5) \Rightarrow (1)$, let $w \in B_p$ and recall that $B_p(w) = B_{p'}(w')$, where $w' = w^{1-p'}$. Hence
\begin{eqnarray*}
\int_{Q_I} M(1_{Q_I}w) & \le &     \left( \int_{Q_I} M(1_{Q_I}w)^{p'} w'   \right)^{1/p'}                 \left( \int_{Q_I}     w     \right)^\frac{1}{p}   \\
&\le &     \| M(w \cdot) \|_{L^{p'}(w) \to L^{p'}(w')}  w(Q_I)^{1/p'} w(Q_I)^{1/p} \\
& =   & \|  M \|_{L^{p'}(w') \to L^{p'}(w')}  w(Q_I) \le B_p(w)^{\frac{1}{p'-1}} w(Q_I),
\end{eqnarray*}
where we have used the estimate (4.7) from \cite{MR3110501} for the maximal function in the last line.

The implication $(5) \Rightarrow (6)$ is a consequence of the fact that if  $w\in B_{p}$, then also $w\in B_{q}$ for any $q>p$, and the limit of the $B_{p}(w)$ as $p\to \infty$ is precisely the expression in (6). Finally, the proof of $(6) \Rightarrow (1)$ can be found in \cite{MR3473651}, as Carleson boxes form a Muckenhoupt bases, and the maximal function associated to it satisfies $L^{p}$ bounds. 

\end{proof}

\section{Further characterizations and applications} \label{s.apply}

We relate $B_\infty$ to the more general classes $B_p(\eta)$  defined as follows.  A measurable positive function $w$,  
belongs to the class $B_p(\eta)$ for $1<p<\infty$, $\eta>-1$,  if and only if

\begin{equation}
\label{eta.bekolle}
B_{p}(w,\eta):=\sup_{\substack{I \text{ interval}\\ I\subset \T}}\left(\frac{1}{A_\eta(Q_I)}\int_{Q_{I}}wdA_\eta\right) \left(\frac{1}{A_\eta(Q_{I})}\int_{Q_{I}}w^{1-p'}dA_\eta\right)^{p-1}<\infty,
\end{equation} 
where $dA_\eta=(1-|z|^2)^{\eta}dA$.
It is a result of  Bekoll\'e \cite{MR667319} that  
$$
\frac{w(z)}{(1-|z|^2)^{\eta}}\in B_p(\eta) \text{ if and only if } P_{\eta}: L^{p}(w)\mapsto L^{p}_{a}(w),
$$
where $P_{\eta}$ is defined as
$$
P_{\eta}f(z)= \int_{\mathbb D}\frac{f(\xi)}{(1-\bar{\xi}z)^{\eta+2}}(1+\eta)(1-|\xi|^{2})^\eta dA(\xi).
$$
The result in \cite{MR667319} is actually stronger. If  $\frac{w}{(1-|z|^2)^\eta}\in B_p(\eta)$, then also the  maximal version  of $P_\eta$,
\begin{equation}\label{p+}
P_{\eta}^+f(z)= \int_{\mathbb D}\frac{f(\xi)}{|1-\bar{\xi}z|^{\eta+2}}(1+\eta)(1-|\xi|^{2})^\eta dA(\xi),
\end{equation}
defines a bounded operator from $L^p(w)$ into itself.

Clearly, if  $\frac{w}{(1-|z|^2)^\eta}\in B_p(\eta),~p>1,\eta>-1$
then $\frac{w}{(1-|z|^2)^\delta}\in B_q(\delta),$ whenever $q\ge p,~\delta\ge \eta$. In the opposite direction, we have the following result.
\begin{lemma} \label{hoelder} Let $1<p<\infty$ and $\eta>\delta>-1$. If $\frac{w}{(1-|z|^2)^\eta}\in B_p(\eta)$ then there exists $q\in (1,\infty)$ such that $\frac{w}{(1-|z|^2)^\delta}\in B_q(\delta)$.
\end{lemma}
\begin{proof}
By H\"older's inequality we have for $q\in (p,\infty)$
$$\int_{Q_I}w^{1-q'}(1-|z|^2)^{\delta q'}dA\le \left(\int_{Q_I}w^{1-p'}(1-|z|^2)^{\eta p'}dA
\right)^{\frac{p-1}{q-1}}\left(\int_{Q_I}(1-|z|^2)^{-\eta p'\frac{q'-1}{p'-q'}+\delta q'\frac{p'-1}{p'-q'}}dA
\right)^{\frac{p'-q'}{p'-1}}.$$
If $q$ is sufficiently large, then 
$$\left(\int_{Q_I}(1-|z|^2)^{-\eta p'\frac{q'-1}{p'-q'}+\delta q'\frac{p'-1}{p'-q'}}dA
\right)^{\frac{p'-q'}{p'-1}}\lesssim |I|^{-\eta p'\frac{q'-1}{p'-1}+\delta q'+2\frac{p'-q'}{p'-1}},$$
and a direct calculation leads to 
$$\left(\frac1{A_\delta(Q_I)}\int_{Q_I}w^{1-q'}(1-|z|^2)^{\delta q'}dA\right)^{q-1} \le \frac{A_\delta(Q_I)}{A_\eta(Q_I)} \left(\frac1{A_\delta(Q_I)}\int_{Q_I}w^{1-p'}(1-|z|^2)^{\eta p'}dA\right)^{p-1},$$
which finishes the proof.
\end{proof}

 Theorem \ref{t.main} and  the above remarks, together with some existing results, yield  the following addtional characterizations of $B_\infty$.

\begin{corollary}\label{c.main}
	 Let $w$ be a weight satisfying \eqref{e.regcon1}. Then the following are equivalent:
	 \begin{itemize}
	 	\item [(a)] $w\in B_\infty$.
	 	
	\item  [(b)]  There exist $p>1$ and $\eta>-1$ such that $\frac{w}{(1-|z|^2)^\eta}\in B_p(\eta)$.
	
	\item [(c) ]There exists $\gamma>0$ such that
	$$\int_\D\frac{w(\xi)}{|1-\bar{\xi}z|^{\gamma+2}}dA(\xi)\lesssim \frac{w(z)}{(1-|z|^2)^\gamma},\quad  z\in \D.$$    
	 		 	
	\item [(d)]  For all $q>0$ and all analytic functions $f$ in $\D$,  $$ \|f\|_{L^q(w)}\approx \sum_{j=0}^{n-1}|f^{j)}(0)|+ \left (\int_{\mathbb D}|f^{n)}|^{q}(1-|z|)^{qn} w dA\right)^{1/q}$$	
	 	
	 	\end{itemize}
	\end{corollary}

\begin{proof} If  (a) holds then (b) follows by Theorem \ref{t.main} ( (1) $\Leftrightarrow$ (5)). Conversely, if (b) holds
	with $p>1$ and $-1<\eta\le 0$, then by the simple observation preceding Lemma \ref{hoelder} we have $w\in B_p$ and  Theorem \ref{t.main} ( (1) $\Leftrightarrow$ (5)) gives (a). If  (b) holds
	with $p>1$ and $\eta> 0$, we can apply  Lemma  \ref{hoelder} to  conclude that $w\in B_q$ for some $q>1$ and (a) follows as above.  (b) $\Rightarrow (c)$  was actually proved in \cite{MR2370047}. We sketch an argument for the sake of completion. If (b) holds,  it follows by  (a) together with Theorem \ref{t.main} that $w\in B_q$ for some $q>1$. Moreover,  by the  result in \cite{MR667319} we have that  the operator $P_0^+:L^q(w)\to L^q(w)$ is bounded.\\ Given $z\in \D$, let $$\Delta_z=\{\xi:~|z-\xi|<\frac{1-|z|}{2}\},$$
and denote by $\chi_z$ its characteristic function.  If $\xi\in \Delta_z$  and $\zeta\in\D$  then
$$\frac{|1-\bar{\xi}\zeta|}{|1-\bar{z}\zeta|}\le 1+\frac{|\xi-z|}{1-|z|}<\frac3{2},$$
and similarly,
$$\frac{|1-\bar{\xi}\zeta|}{|1-\bar{z}\zeta|}\ge 1-\frac{|\xi-z|}{1-|z|}>\frac1{2}.$$
Thus
 $$P_0^+\chi_z(\zeta)\sim \frac{A(\Delta_z)}{|1-\bar{\zeta}z|^{2}}, $$
and
$$\|P_0^+\chi_z\|^q_{L^q(w)}\sim (A(\Delta_z))^q\int_\D\frac{w(\zeta)}{|1-\bar{\zeta}z|^{2q}}dA(\zeta).$$
On the other hand, $\Delta_z$ is contained in a top half $T_{I,\rho}$,  hence by \eqref{e.regcon1}  it follows that 
$$\|\chi_z\|^q_{L^q(w)}\sim w(z)A(\Delta_z).$$
Then (c) follows directly from $$\|P_0^+\chi_z\|_{L^q(w)}\le C\|\chi_z\|_{L^q(w)}.$$
Assume that (c) holds and let $I$ be an arc on $\T$. It is a simple exercise to show that for $z,\xi\in Q_I$ we have 
$$|1-\bar{\xi}z|\sim |Q_I|^{1/2}=(A_\gamma(Q_I))^{\frac1{\gamma+2}}.$$
The (c) gives for $z\in Q_I$
$$  \int_{Q_I}wdA  \lesssim A_\gamma(Q_I)\frac{w(z)}{(1-|z|^2)^\gamma} ,$$
hence for every $p>1$, 
$$\left(\frac{w(z)}{(1-|z|^2)^\gamma}\right)^{1-p'}\left(\int_{Q_I}wdA\right)^{p'-1}\lesssim (A_\gamma(Q_I))^{p'-1},\quad z\in Q_I.$$
Thus 
$$\left(\int_{Q_I}wdA\right)^{p'-1}\int_{Q_I}\left(\frac{w(z)}{(1-|z|^2)^\gamma}\right)^{1-p'}dA_\gamma(z)\lesssim (A_\gamma(Q_I))^{p'},$$
 which shows that  $\frac{w}{(1-|z|^2)^\gamma}\in B_p(\gamma)$ for all $p>1$.\\
 (b)$\Leftrightarrow$ (d) is proved in [\cite{MR2585394}, Theorem 3.2] for differentiable weights satisfying the condition in Proposition \ref{p.regcond}
 (a). Using  Proposition \ref{p.regcond}
 (b) we see that the equivalence holds for all weights satisfying  \eqref{e.regcon1}.
\end{proof}
Finally we mention an application concerning integral operators of the form
\begin{equation}\label{e.intop}
T_{g}f(z)=\int_{0}^{z}f(\xi)g'(\xi)d\xi.
\end{equation}
on the weighted Bergman spaces $L^{p}_{a}(w)=L^p(w)\cap Hol(\D),~1\le p<\infty$, where $w$ is a weight on $\mathbb D$  satisfying \eqref{e.regcon1}.
There is a vast lierature on the subject (see \cite{MR2585394} and the references therein.) The results in \cite{MR2585394} are proved for differentiable weights satisfying the condition in Proposition \ref{p.regcond}
(a), hence by  Proposition \ref{p.regcond} (b) they continue to hold for all weights satisfying  \eqref{e.regcon1}. For example, $T_g$ is bounded on $L_a^p(w)$ if and only if the symbol $g$ belongs to the Bloch space, that is $$|g'(z)|\lesssim\frac1{1-|z|}.$$ 
Using Corollary \ref{c.main} (and Proposition \ref{p.regcond} (b)) the description of the spectrum of $T_g$ provided by Theorem 5.1 in  \cite{MR2585394} can be reformulated as follows.
\begin{corollary}\label{tg}
A point $\lambda\in \mathbb{C}\setminus\{0\}$ belongs to the resolvent set of $T_g$ on $L_a^p(w)$ if and only if $w\exp(p\text{Re}\frac{g}{\lambda})\in B_\infty$.
\end{corollary}
This illustrates again the analogy to $A_\infty$, since for Hardy spaces the spectrum of $T_g$ is described in the same manner using $A_\infty$ instead (see \cite{MR3043595}).

\section{Acknowledgement}
We thank the Swedish Agency for Innovation, VINNOVA, for the partial support provided to carry out  this research through its Marie Curie Incoming project number 2014-01434 with title  
``Dyadisk harmonisk analys och viktad teori i Bergmanrummet''. The second author was also supported by VR grant 2015-05552.



\end{document}